\documentclass[12pt,a4paper,twoside,english,american]{scrartcl}
\usepackage{helvet}
\usepackage[T1]{fontenc}
\usepackage[latin9]{inputenc}
\usepackage{textcomp}
\usepackage{amsmath}
\usepackage{amssymb}

\makeatletter

\newcommand{\lyxmathsym}[1]{\ifmmode\begingroup\def\b@ld{bold}
  \text{\ifx\math@version\b@ld\bfseries\fi#1}\endgroup\else#1\fi}

 \usepackage{amsthm}
 \theoremstyle{plain}
\newtheorem{thm}{Theorem}[section]
  \theoremstyle{remark}
  \newtheorem{rem}[thm]{Remark}
  \theoremstyle{definition}
  \newtheorem{defn}[thm]{Definition}
 \theoremstyle{definition}
  \newtheorem{example}[thm]{Example}
  \theoremstyle{plain}
  \newtheorem{cor}[thm]{Corollary}

\usepackage[T1]{fontenc}    

\usepackage{a4wide}        
\usepackage{tu-preprint}
\usepackage{amsmath}

\usepackage{euscript}
\usepackage{mathabx}
\allowdisplaybreaks[1] 
\usepackage{amsfonts}
\newenvironment{keywords}{ \noindent\footnotesize\textbf{Keywords and phrases:}}{}

\newenvironment{class}{\noindent\footnotesize\textbf{Mathematics subject classification 2000:}}{}

\usepackage{color,tu-preprint}

\newcommand*{\dive}{\operatorname{div}}

\newcommand*{\curl}{\operatorname{curl}}

\newcommand*{\grad}{\operatorname{grad}}

\renewcommand*{\i}{\mathrm{i}}

\DeclareMathAccent{\Circ}{\mathalpha}{operators}{"17}
\newcommand{\interior}[1]{\Circ{#1}}
\renewcommand{\Im}{\operatorname{\mathfrak{Im}}}
\renewcommand{\Re}{\operatorname{\mathfrak{Re}}}

\renewcommand*{\epsilon}{\varepsilon}
\renewcommand*{\theta}{\vartheta}
\renewcommand*{\nu}{\varrho}

\makeatother

\usepackage{babel}
\begin{document}
\selectlanguage{english}%
\institut{Institut f\"ur Analysis}

\preprintnumber{MATH-AN-03-2012}

\preprinttitle{The Elusive Drude-Born-Fedorov Model for Chiral Electromagnetic Media.}

\author{Henrik Freymond and Rainer Picard}

\makepreprinttitlepage

\selectlanguage{american}%
\setcounter{section}{-1}

\date{}

\title{The Elusive Drude-Born-Fedorov Model for Chiral Electromagnetic Media.}

\author{Henrik Freymond \& Rainer Picard\\
Institut für Analysis,Fachrichtung Mathematik\\
 Technische Universität Dresden\\
 Germany\\
 rainer.picard@tu-dresden.de }
\maketitle
\begin{abstract}
Electro-magnetic wave propagation in more complex linear materials
such as bi-anisotropic media have come to a considerable attention
within the last fifteen to twenty years. The Drude-Born-Fedorov model
has been extensively studied mostly in the time-harmonic case as a
model for chiral media. In the physically relevant time-dependent
case the record is much less convincing. In this paper we focus on
this case and analyze the Drude-Born-Fedorov model in the light of
recently developed Hilbert space approach to evolutionary problems.
The solution theory will be developed in the framework of extrapolation
spaces (Sobolev lattices).
\end{abstract}

\begin{class}
35Q61, 78A25
\end{class}

\tableofcontents{}

\section{Introduction}

The evolutionary system of the macroscopic Maxwell's equations 
\begin{eqnarray*}
\textrm{curl }H-\partial_{0}D & = & j\,,\\
\textrm{curl }E+\partial_{0}B & = & 0\,,
\end{eqnarray*}
complemented by suitable boundary conditions such as the vanishing
of the tangential component of the electric field $E$ and initial
data for $D$ and $B$ is completed by a material law connecting the
combined six component vector field $\left(D,B\right)$ composed from
the electric displacement current $D$ and the magnetic induction
$B$ with the electro-magnetic field $\left(E,H\right)$ ($\partial_{0}$
denotes the time derivative). In the simplest case it is assumed that
\begin{eqnarray*}
D & = & \epsilon\, E\\
B & = & \mu\, H
\end{eqnarray*}
with constant coefficient $\epsilon,\mu\in\mathbb{R}_{>0}$ (permittivity
$\epsilon$, permeability $\mu$). This case is referred to as the
case of isotropic and homogeneous media. There is a long history of
various description of more complex electromagnetic media and a shorter
one for their mathematical treatment. As -- in a sense -- the conclusion
of the treatment of the anisotropic and inhomogeneous media case we
refer to the presentation of the Maxwell system in the book by R.
Leis \cite{Leis:Buch:2} from 1986 and the references given there.
The functional analytical approach utilized shows that indeed a mathematical
solution theory can be obtained for rather general media merely requiring
that $\epsilon,\mu$ are time-independent, bounded, self-adjoint,
strictly positive mappings in $L^{2}$-type spaces of vector fields
in an non-empty open set $\Omega\subseteq\mathbb{R}^{3}$ containing
the media characterized by $\epsilon$ and $\mu$, which covers the
case of real multiplicative matrix-valued operations with in $\Omega$
$L^{\infty}$-bounded entries as a special case. Although, the necessary
concepts are -- as one says -- well-known since the late seventies
of last century, it seems not too widely adopted that by appropriately
generalizing the concept of assumption of boundary conditions boundary
singularities such as corners, edges, cusps, even fractal boundaries
can be included. The price to be paid for this is a more subtle approach
exploiting the full power of functional analytic concepts (as a general
reference for functional analytical concepts see \cite{Kato,Yosida_6th}).
As it turns out, however, this can be achieved purely in a Hilbert
space setting, thus reducing conveniently the conceptual complexity
of our considerations. 

Writing $\left\langle \:\cdot\:|\:\cdot\:\right\rangle _{0}$ for
an $L^{2}$-type inner product regardless of the number of components,
we have 
\[
\left\langle E\Big|H\right\rangle _{0}=\sum_{k=1}^{3}\left\langle E_{k}\Big|H_{k}\right\rangle _{0}
\]
for vector fields 
\[
E=\left(E_{1},E_{2},E_{3}\right)=\left(\begin{array}{c}
E_{1}\\
E_{2}\\
E_{3}
\end{array}\right),\: H=\left(H_{1},H_{2},H_{3}\right)=\left(\begin{array}{c}
H_{1}\\
H_{2}\\
H_{3}
\end{array}\right)
\]
in $L^{2}\left(\Omega\right)^{3}$. For ease of notation we shall
rarely note the number of components and write e.g. $L^{2}\left(\Omega\right)$
instead of $L^{2}\left(\Omega\right)^{3}$, since the number of components
will always be clear from the context. It is clear from the arguments
e.g. in \cite{0579.58030}, although the impact and relevance of this
observation for the more recent discussion of electromagnetic meta-materials
could not be realized at the time, that the mathematical methods developed
in connection with anisotropic, inhomogeneous media extend to the
case of material laws of the block matrix form
\[
\left(\begin{array}{c}
D\\
B
\end{array}\right)=M_{0}\left(\begin{array}{c}
E\\
H
\end{array}\right)
\]
with
\[
M_{0}=\left(\begin{array}{cc}
\epsilon & \beta\\
\beta^{*} & \mu
\end{array}\right)
\]
in an obvious way, provided that the entries $\epsilon,\,\mu,\,\beta$
are time-independent, bounded, linear mappings in $L^{2}$-type spaces
of vector fields such that the block matrix operator
\begin{eqnarray*}
M_{0}:L^{2}\left(\Omega\right) & \to & L^{2}\left(\Omega\right)\\
\left(\begin{array}{c}
E\\
H
\end{array}\right) & \mapsto & \left(\begin{array}{cc}
\epsilon & \beta\\
\beta^{*} & \mu
\end{array}\right)\left(\begin{array}{c}
E\\
H
\end{array}\right)
\end{eqnarray*}
is still self-adjoint and strictly positive.%
\footnote{This is the case if and only if $\epsilon,\mu-\beta^{*}\epsilon^{-1}\beta$
are selfadjoint and strictly positive, which amounts to a smallness
condition on the off-diagonal entry $\beta$.%
}  

Based on a structural observation described in \cite{2009-2} we shall
here consider in the following a class of well-posed evolutionary
equations in even more complex media. Such media have become of strong
interest in recent decades due to their prospects in view of their
interesting physical properties (for a survey we refer the reader
to \cite{LSTV,0866.00023}). This problem class, which conveniently
contains materials with memory, will lead us to the so-called Drude-Born-Fedorov
model for chiral media, which has been initially introduced as an
approximation of the actual material description in an attempt to
avoid non-local-in-time terms (optical response).

The model starts from a material relation of the form
\[
D=\epsilon\left(E+\eta\:\curl E\right),\: B=\mu\left(H+\eta\:\curl H\right),
\]
where $\epsilon,\mu\in\mathbb{R}_{>0}$ and $\eta\:\in\mathbb{R}$
is the so-called chirality parameter.

There is a long list of publications dealing with this model. Despite
this fact and contrary to various claims made, it appears that this
model has not been fully understood in sufficient generality. It is
due to this state of affairs that we have chosen to speak of the {}``elusive''
Drude-Born-Fedorov model. So has it only more recently been understood
that the model has to be considered in the -- much less forgiving
-- dynamic case, \cite{1052.78002,1062.78004,1075.78005,1126.78002,pre05506047},
rather than the so-called time-harmonic case. 

We will show that in the framework of our approach from \cite{2009-2}
the Drude-Born-Fedorov model will have to be considered as a rather
degenerate case, which, however, by some additional considerations
can be elegantly embedded into the standard theory. Indeed, we are
led to a simple evolution equation with a bounded generator. This
has in essence already been observed in \cite{1177.35042} for simply
connected domains with smooth boundary. It has gone unnoticed, however,
that the result can be extended to fairly arbitrary open sets with
a boundary permitting a suitable local compact embedding result, \cite{0926.35104,0935.35029,Filonov00}.
As a by-product we can discuss well-posedness for a class of bi-anisotropic
materials with \emph{modified} Drude-Born-Fedorov type material behavior,
some cases of which have found previous attention, compare \cite[Morro 2002]{A2002285},
\cite[Sj{\"o}berg 2008]{0022-3727-41-15-155412}.

\section{Space-Time Evolution Equations\label{sec:Space-Time-Evolution-Equations} }

\subsection{A brief summary of Sobolev chains and lattices }

We recall first the standard construction of chains of Hilbert spaces
associated with a normal operator $O$ in a Hilbert space $X$ with
$0$ in its resolvent set, see e.g. \cite{PIC_2010:1889}. Defining
$H_{k}\left(O\right)$ as the completion of $D\left(O^{k}\right)$
with respect to the norm $\left|\:\cdot\:\right|_{k}$ induced by
the inner product $\left\langle \:\cdot\:|\:\cdot\:\right\rangle _{k}$
given by 
\[
\left(f,g\right)\mapsto\left\langle O^{k}f|O^{k}g\right\rangle _{X}
\]
for $k\in\mathbb{Z}$, we obtain a chain $\left(H_{k}\left(O\right)\right)_{k\in\mathbb{Z}}$
of Hilbert space, where the order is given by canonical, continuous
and dense embeddings 
\[
H_{k+1}\left(O\right)\hookrightarrow H_{k}\left(O\right),\: k\in\mathbb{Z}.
\]
Note that $H_{1}\left(O\right)$ is the domain of $O$ and for $k\in\mathbb{N}$
\[
H_{k}\left(O\right)\hookrightarrow X\hookrightarrow H_{-k}\left(O\right)
\]
is a Gelfand triple. By construction 
\begin{eqnarray*}
H_{1}\left(O\right) & \to & H_{0}\left(O\right)\\
f & \mapsto & Of
\end{eqnarray*}
is a unitary mapping and it follows by induction that for $k\in\mathbb{N}$
also 
\begin{eqnarray*}
H_{k+1}\left(O\right) & \to & H_{k}\left(O\right)\\
f & \mapsto & Of
\end{eqnarray*}
is unitary. Moreover, $O$ extends to a unitary mapping
\begin{eqnarray*}
H_{k+1}\left(O\right) & \to & H_{k}\left(O\right)\\
f & \mapsto & Of
\end{eqnarray*}
also for negative $k$ and so to all $k\in\mathbb{Z}$. 
\begin{rem}
Since $O=\Re O+\i\Im O$ is a normal operator in $H_{0}\left(O\right)$
we have
\begin{eqnarray*}
O^{*} & = & \Re O-\i\Im O
\end{eqnarray*}
Since $D\left(O\right)=D\left(O^{*}\right)$ and 
\[
\left|Of\right|_{X}=\left|O^{*}f\right|_{X}
\]
that also $O$ has a unitary extension
\begin{eqnarray*}
H_{k+1}\left(O\right) & \to & H_{k}\left(O\right)\\
f & \mapsto & O^{*}f
\end{eqnarray*}
for $k\in\mathbb{Z}$. Keeping the notation $O^{*}$ for this extension,
as we shall do, leads to a possible confusion, since $O^{*}$ could
be misinterpreted as the adjoint, i.e. the inverse of $O$ \emph{as
a unitary mapping}. To avoid this misunderstanding we shall \emph{not}
write the inverse of a unitary operator as its adjoint. 

With this convention we can safely say that we have established $O^{s}$
and $\left(O^{*}\right)^{s}$ as unitary mappings 
\begin{eqnarray*}
H_{k+s}\left(O\right) & \to & H_{k}\left(O\right)\\
f & \mapsto & O^{s}f
\end{eqnarray*}
and
\begin{eqnarray*}
H_{k+s}\left(O\right) & \to & H_{k}\left(O\right)\\
f & \mapsto & \left(O^{*}\right)^{s}f
\end{eqnarray*}
for $s,k\in\mathbb{Z}$.
\end{rem}
We shall utilize this abstract construction for various special cases.
Note that for $O\in L\left(H,H\right)$ we have 
\[
H_{s}\left(O\right)=H
\]
as topological linear spaces with merely different inner products
(inducing equivalent norms) in the different Hilbert spaces $H_{s}\left(O\right)$,
$s\in\mathbb{Z}$. This indicates that considering continuous linear
operators $O$ does not lead to interesting chains.

A particular instance of this construction is if $O$ is the time-derivative
$\partial_{0}$. We recall, e.g. from \cite{2009-2,PIC_2010:1889},
that differentiation considered in the complex Hilbert space $H_{\nu,0}(\mathbb{R}):=\{f\in L_{\textnormal{loc}}^{2}(\mathbb{R})|(x\mapsto\exp(-\nu x)f(x))\in L^{2}(\mathbb{R})\}$,
$\nu\in\mathbb{R}\setminus\left\{ 0\right\} $, with inner product
\[
(f,g)\mapsto\langle f,g\rangle_{\nu,0}:=\int_{\mathbb{R}}f(x)^{*}g(x)\:\exp(-2\nu x)\: dx
\]
can indeed be established as a normal operator, which we denote by
$\partial_{0,\nu}$, with 
\[
\Re\partial_{0,\nu}=\nu.
\]
For $\Im\partial_{0,\nu}$ we have as a spectral representation the
Fourier-Laplace transform $\mathcal{L}_{\nu}:H_{\nu,0}(\mathbb{R})\to L^{2}\left(\mathbb{R}\right)$
given by the unitary extension of 
\begin{align*}
\interior C_{\infty}\left(\mathbb{R}\right)\subseteq H_{\nu,0}(\mathbb{R}) & \to L^{2}(\mathbb{R})\\
\phi & \mapsto\left(x\mapsto\frac{1}{\sqrt{2\pi}}\int_{\mathbb{R}}\exp\left(-\mathrm{i}xy\right)\;\exp\left(-\nu y\right)\phi(y)\; dy\right).
\end{align*}
 In other words, we have the unitary equivalence
\[
\Im\partial_{0,\nu}=\mathcal{L}_{\nu}^{-1}m\:\mathcal{L}_{\nu},
\]
where $m$ denotes the selfadjoint multiplication-by-argument operator
in $L^{2}\left(\mathbb{R}\right)$. Since $0$ is in the resolvent
set of $\partial_{0,\nu}^{-1}$ we have that $\partial_{0,\nu}^{-1}$
is an element of the Banach space $L\left(H_{\nu,0}(\mathbb{R}),H_{\nu,0}(\mathbb{R})\right)$
of continuous (left-total) linear mappings in $H_{\nu,0}(\mathbb{R})$.
Denoting generally the operator norm of the Banach space $L\left(X,Y\right)$
by $\left\Vert \:\cdot\;\right\Vert _{L\left(X,Y\right)}$, we get
for $\partial_{0}^{-1}$
\[
\left\Vert \partial_{0,\nu}^{-1}\right\Vert _{L\left(H_{\nu,0}(\mathbb{R}),H_{\nu,0}(\mathbb{R})\right)}=\frac{1}{\nu}.
\]
Not too surprisingly, we find for $\nu>0$ 
\[
\left(\partial_{0,\nu}^{-1}\varphi\right)\left(x\right)=\int_{-\infty}^{x}\varphi\left(t\right)\: dt
\]
and for $\nu<0$
\[
\left(\partial_{0,\nu}^{-1}\varphi\right)\left(x\right)=-\int_{x}^{\infty}\varphi\left(t\right)\: dt
\]
for all $\varphi\in\interior C_{\infty}\left(\mathbb{R}\right)$ and
$x\in\mathbb{R}$. Since we are interested in the forward causal situation,
we assume $\nu>0$ throughout. Moreover, in the following we shall
mostly write $\partial_{0}$ for $\partial_{0,\nu}$ if the choice
of $\nu$ is clear from the context. 

Thus, with $O=\partial_{0,\nu}$ in the above general construction
we obtain a chain $\left(H_{\nu,k}\left(\mathbb{R}\right)\right)_{k\in\mathbb{Z}}$
of Hilbert spaces
\[
H_{\nu,k}\left(\mathbb{R}\right)\::=H_{k}\left(\partial_{0,\nu}\right).
\]
Similarly, with $O=\i m+\nu$ in $L^{2}\left(\mathbb{R}\right)$ we
obtain the chain of polynomially weighted $L^{2}\left(\mathbb{R}\right)$-spaces
\[
\left(L_{k}^{2}\left(\mathbb{R}\right)\right)_{k\in\mathbb{Z}}
\]
with
\[
L_{k}^{2}\left(\mathbb{R}\right)\::=\left\{ f\in L^{2,\mathrm{loc}}\left(\mathbb{R}\right)|\,\left(\i m+\nu\right)^{k}f\in L^{2}\left(\mathbb{R}\right)\right\} =H_{k}\left(\i m+\nu\right)
\]
for $k\in\mathbb{Z}$.

Since the unitarily equivalent operators $\partial_{0,\nu}$ and $\i m+\nu$
can canonically be lifted to the $X$-valued case, $X$ an arbitrary
complex Hilbert space, we are lead to a corresponding chain $\left(H_{\nu,k}\left(\mathbb{R},X\right)\right)_{k\in\mathbb{Z}}$
and $\left(L_{k}^{2}\left(\mathbb{R},X\right)\right)_{k\in\mathbb{Z}}$
of $X$-valued generalized functions. The Fourier-Laplace transform
can also be lifted to the $X$-valued case yielding
\begin{eqnarray*}
H_{\nu,k}\left(\mathbb{R},X\right) & \to & L_{k}^{2}\left(\mathbb{R},X\right)\\
f & \mapsto & \mathcal{L}_{\nu}f
\end{eqnarray*}
as a unitary mapping for $k\in\mathbb{N}$ and by continuous extension,
keeping the notation $\mathcal{L}_{\nu}$ for the extension, also
for $k\in\mathbb{Z}$. Since $\mathcal{L}_{\nu}$ has been constructed
from a spectral representation of $\Im\partial_{0,\nu}$, we can utilize
the corresponding operator function calculus for functions of $\Im\partial_{0,\nu}$.
Noting that $\partial_{0,\nu}=\i\Im\partial_{0,\nu}+\nu$ is a function
of $\Im\partial_{0,\nu}$ we can define operator-valued functions
of $\partial_{0}.$
\begin{defn}
Let $r>\frac{1}{2\nu}>0$ and $M:B_{\mathbb{C}}(r,r)\to L(H,H)$ be
bounded and analytic, $H$ a Hilbert space. Then define 
\[
M\left(\partial_{0}^{-1}\right):=\mathbb{L}_{\nu}^{*}\: M\left(\frac{1}{\mathrm{i}m+\nu}\right)\:\mathbb{L}_{\nu},
\]
 where 
\[
M\left(\frac{1}{\mathrm{i}m+\nu}\right)\phi(t):=M\left(\frac{1}{\mathrm{i}t+\nu}\right)\phi(t)\quad(t\in\mathbb{R})
\]
for $\phi\in\interior C_{\infty}\left(\mathbb{R},H\right)$. \end{defn}
\begin{rem}
\label{Rem: rho-Independent} The definition of $M(\partial_{0}^{-1})$
is largely independent of the choice of $\nu$ in the sense that the
operators for two different parameters $\nu_{1},\nu_{2}$ coincide
on the intersection of the respective domains. 
\end{rem}
Simple examples are polynomials in $\partial_{0}^{-1}$ with operator
coefficients. A more exotic example of an analytic and bounded function
of $\partial_{0}^{-1}$ is the delay operator, which itself is a special
case of the time translation: 
\begin{example}
Let $r>0$, $\nu>\frac{1}{2r}$, $h\in\mathbb{R}$ and $u\in H_{\nu,0}(\mathbb{R},X)$.
We define 
\[
\tau_{h}u:=u(\:\cdot\:+h).
\]
 The operator $\tau_{h}\in L(H_{\nu,0}(\mathbb{R},X),H_{\nu,0}(\mathbb{R},X))$
is called a \emph{time-translation operator}. If $h<0$ the operator
$\tau_{h}$ is also called a \emph{delay operator}. In the latter
case the function 
\[
B_{\mathbb{C}}(r,r)\ni z\mapsto M(z):=\exp(z^{-1}h)
\]
 is analytic and uniformly bounded for every $r\in\mathbb{R}_{>0}$
(considered as an $L\left(X,X\right)$-valued function). An easy computation
shows for $u\in H_{\nu,0}\left(\mathbb{R},X\right)$ that 
\[
u(\:\cdot\:+h)=\mathbb{L}_{\nu}^{*}\exp((\mathrm{i}m+\nu)h)\mathbb{L}_{\nu}u=M(\partial_{0}^{-1})u=\exp(\left(\partial_{0}^{-1}\right)^{-1}h)\: u.
\]
This shows that
\[
\tau_{h}=\exp\left(h/\partial_{0}^{-1}\right)=\exp\left(h\partial_{0}\right).
\]

\end{example}
Another class of interesting bounded analytic functions of $\partial_{0}^{-1}$
are mappings produced by a temporal convolution with a suitable operator-valued
integral kernel. 

Let now $O$ denote a normal operator in Hilbert space $H$ with $0$
in the resolvent set. Then $O$ has a canonical extension to the time-dependent
case i.e. to $H_{\nu,0}\left(\mathbb{R},H\right)$. Then $\partial_{0}$
and $O$ become commuting normal operators and by combining the two
chains we obtain a Sobolev lattice in the sense of \cite{PDE_DeGruyter}
based on $\left(\partial_{0},O\right)$ yielding a family of Hilbert
spaces ($\nu\in\mathbb{R}\setminus\left\{ 0\right\} $) 
\[
\left(H_{\nu,k}\left(\mathbb{R},\: H_{s}\left(O\right)\right)\right)_{k,s\in\mathbb{Z}}
\]
 with norms $\left|\:\cdot\:\right|_{\nu,k,s}$ given by
\[
v\mapsto\left|\partial_{0}^{k}O^{s}v\right|_{H_{\nu,0}\left(\mathbb{R},\: H\right)}
\]
for $k,s\in\mathbb{Z}$. Note that
\[
H_{0}\left(O\right)=H
\]
independent of the particular choice of $O$.

\subsection{Abstract initial value problems}

We shall discuss equations of the form
\begin{equation}
\left(\partial_{0}M\left(\partial_{0}^{-1}\right)+A\right)U=\mathcal{J}+\delta\otimes W_{0},\label{eq:evo-abstract}
\end{equation}
where for simplicity we shall assume that $A$ is skew-selfadjoint
in a Hilbert space $H$ and $M$ is a regular material law in the
sense of \cite{Pi2009-1,PDE_DeGruyter}. More specifically we assume
that $M$ is of the form 
\[
M\left(z\right)=M_{0}+zM_{1}\left(z\right)
\]
where $M_{1}$ is an analytic and bounded $L\left(H,H\right)$-valued
function in a ball $B_{\mathbb{C}}\left(r,r\right)$ for some $r\in\mathbb{R}_{>0}$
and $M_{0}$ is a continuous, selfadjoint and strictly positive definite
operator in $H$. The operator $M\left(\partial_{0}^{-1}\right)$
is then to be understood in the sense of the operator-valued function
calculus associated with the selfadjoint operator $\Im\left(\partial_{0}\right)=\frac{1}{2\i}\left(\partial_{0}+\partial_{0}^{*}\right)$. 

For the data we assume
\[
\mathcal{J}\in H_{\nu,0}\left(\mathbb{R},\: H\right),\:\mathcal{J}=0\mbox{ on }\mathbb{R}_{<0},
\]
 and 
\[
W_{0}\in H,
\]
which makes (\ref{eq:evo-abstract}) an abstract initial value problem.
The appropriate setting turns out to be the Sobolev lattice 
\[
\left(H_{\nu,k}\left(\mathbb{R},\: H_{s}\left(\sqrt{M_{0}^{-1}}A\sqrt{M_{0}^{-1}}+1\right)\right)\right)_{k,s\in\mathbb{Z}},
\]
where, however, only the spaces with $s=-1,0,1$ and $k=-2,-1,0,\,1$
are actually utilized. From \cite{Pi2009-1,PDE_DeGruyter} we paraphrase
the following solution result on which our approach to the Drude-Born-Fedorov
model can conveniently be based. 
\begin{thm}
\label{SolutionTheory}The abstract initial value problem (\ref{eq:evo-abstract})
has a unique solution $U\in H_{\nu,-1}\left(\mathbb{R},H\right)$.
Moreover, 
\[
F\mapsto\left(\partial_{0}M\left(\partial_{0}^{-1}\right)+A\right)^{-1}F
\]
is a linear mapping in $L\left(H_{\nu,k}\left(\mathbb{R},\: H\right),H_{\nu,k}\left(\mathbb{R},\: H\right)\right),$
$k\in\mathbb{Z}.$ These mappings are causal in the sense that if
$F\in H_{\nu,k}\left(\mathbb{R},\: H\right)$ vanishes on the time
interval $]-\infty,\, a]$, then so does \textup{$\left(\partial_{0}M\left(\partial_{0}^{-1}\right)+A\right)^{-1}F$,
$a\in\mathbb{R}$, $k\in\mathbb{Z}$. In particular, we have
\[
U=0\mbox{ on }\mathbb{R}_{<0}.
\]
}
\end{thm}
To link up with a more classical interpretation of the assumption
of the initial data we record the following regularity result.
\begin{thm}
\label{RegRes}Let $U\in H_{\nu,-1}\left(\mathbb{R},H\right)$ be
the unique solution of the abstract initial value problem (\ref{eq:evo-abstract}).
Then, we also have
\begin{equation}
U\in H_{\nu,0}\left(\mathbb{R},H\right)\label{eq:cont}
\end{equation}
and $U$ is a continuous%
\footnote{In the usual sense of having a continuous representer.%
} $H$-valued function on $\mathbb{R}\setminus\left\{ 0\right\} $.
Moreover, 
\begin{equation}
U\left(0+\right)=M_{0}^{-1}W_{0},\label{eq:ic}
\end{equation}
where the limit is taken in $H.$\end{thm}
\begin{proof}
These stronger regularity statements will rely on the function calculus
for the skew-selfadjoint operator $\sqrt{M_{0}^{-1}}A\sqrt{M_{0}^{-1}}$
or -- depending on the point of view -- on one-parameter semi-group
arguments. Let $U\in H_{\nu,-1}\left(\mathbb{R},H\right)$ be the
solution of (\ref{eq:evo-abstract}). Noting that
\[
\left(\partial_{0}M_{0}+A\right)^{-1}=\sqrt{M_{0}^{-1}}\left(\partial_{0}+\sqrt{M_{0}^{-1}}A\sqrt{M_{0}^{-1}}\right)^{-1}\sqrt{M_{0}^{-1}}
\]
we have
\[
\left(\partial_{0}+\sqrt{M_{0}^{-1}}M_{1}\left(\partial_{0}^{-1}\right)\sqrt{M_{0}^{-1}}+\sqrt{M_{0}^{-1}}A\sqrt{M_{0}^{-1}}\right)\sqrt{M_{0}}U=\sqrt{M_{0}^{-1}}\mathcal{J}+\delta\otimes\sqrt{M_{0}^{-1}}W_{0}.
\]
From this we have
\begin{eqnarray*}
\left(1+\left(\partial_{0}+\sqrt{M_{0}^{-1}}A\sqrt{M_{0}^{-1}}\right)^{-1}M_{1}\left(\partial_{0}^{-1}\right)\right)\sqrt{M_{0}}U & = & \left(\partial_{0}+\sqrt{M_{0}^{-1}}A\sqrt{M_{0}^{-1}}\right)^{-1}\sqrt{M_{0}^{-1}}\mathcal{J}+\\
 &  & +\left(\partial_{0}+\sqrt{M_{0}^{-1}}A\sqrt{M_{0}^{-1}}\right)^{-1}\delta\otimes\sqrt{M_{0}^{-1}}W_{0}.
\end{eqnarray*}
We have 
\[
\left(\partial_{0}+\sqrt{M_{0}^{-1}}A\sqrt{M_{0}^{-1}}\right)^{-1}\delta\otimes\sqrt{M_{0}^{-1}}W_{0}=\left(t\mapsto\chi_{_{\mathbb{R}_{\geq0}}}\left(t\right)\:\exp\left(-t\sqrt{M_{0}^{-1}}A\sqrt{M_{0}^{-1}}\right)\sqrt{M_{0}^{-1}}W_{0}\right)
\]
by uniqueness of solution, which as we read off is in $H_{\nu,0}\left(\mathbb{R},H\right)$
and is a continuous function on $\mathbb{R}\setminus\left\{ 0\right\} $.
In particular, we also have
\begin{equation}
\left(\left(\partial_{0}+\sqrt{M_{0}^{-1}}A\sqrt{M_{0}^{-1}}\right)^{-1}\delta\otimes\sqrt{M_{0}^{-1}}W_{0}\right)\left(0+\right)=\sqrt{M_{0}^{-1}}W_{0}\label{eq:idata}
\end{equation}
in $H$. Moreover, for $\nu\in\mathbb{R}_{>0}$ sufficiently large
we have that $\left(\partial_{0}+\sqrt{M_{0}^{-1}}A\sqrt{M_{0}^{-1}}\right)^{-1}M_{1}\left(\partial_{0}^{-1}\right)$
is a contraction since
\[
\left|\left(\partial_{0}+\sqrt{M_{0}^{-1}}A\sqrt{M_{0}^{-1}}\right)^{-1}f\right|_{\nu,0,0}\leq\frac{1}{\nu}\left|f\right|_{\nu,0,0}.
\]
Therefore 
\[
U=\sqrt{M_{0}^{-1}}\left(1-Q\right)^{-1}\left(\partial_{0}+\sqrt{M_{0}^{-1}}A\sqrt{M_{0}^{-1}}\right)^{-1}\left(\sqrt{M_{0}^{-1}}\mathcal{J}+\delta\otimes\sqrt{M_{0}^{-1}}W_{0}\right)\in H_{\nu,0}\left(\mathbb{R},H\right),
\]
where $Q=-\left(\partial_{0}+\sqrt{M_{0}^{-1}}A\sqrt{M_{0}^{-1}}\right)^{-1}M_{1}\left(\partial_{0}^{-1}\right).$
We also have that 
\begin{eqnarray*}
 &  & \left(1+\left(\partial_{0}+\sqrt{M_{0}^{-1}}A\sqrt{M_{0}^{-1}}\right)^{-1}M_{1}\left(\partial_{0}^{-1}\right)\right)^{-1}=\\
 &  & =1-\left(\partial_{0}+\sqrt{M_{0}^{-1}}A\sqrt{M_{0}^{-1}}\right)^{-1}M_{1}\left(\partial_{0}^{-1}\right)\left(1+\left(\partial_{0}+\sqrt{M_{0}^{-1}}A\sqrt{M_{0}^{-1}}\right)^{-1}M_{1}\left(\partial_{0}^{-1}\right)\right)^{-1}
\end{eqnarray*}
and so
\begin{eqnarray}
U & = & \sqrt{M_{0}^{-1}}\left(\partial_{0}+\sqrt{M_{0}^{-1}}A\sqrt{M_{0}^{-1}}\right)^{-1}\delta\otimes\sqrt{M_{0}^{-1}}W_{0}+\label{eq:ireg}\\
 &  & +\sqrt{M_{0}^{-1}}\left(\partial_{0}+\sqrt{M_{0}^{-1}}A\sqrt{M_{0}^{-1}}\right)^{-1}F\nonumber 
\end{eqnarray}
for some $F\in H_{\nu,0}\left(\mathbb{R},H\right)$ with $F=0$ on
$\mathbb{R}_{<0}.$ Since 
\begin{eqnarray*}
\left|\int_{-\infty}^{t}\exp\left(-\left(t-s\right)\sqrt{M_{0}^{-1}}A\sqrt{M_{0}^{-1}}\right)F\left(s\right)\: ds\right| & \leq & \int_{-\infty}^{t}\left|F\left(s\right)\right|\: ds\\
 &  & =\int_{-\infty}^{t}\exp\left(\nu s\right)\,\left|F\left(s\right)\right|\,\exp\left(-\nu s\right)\: ds,\\
 & \leq & \frac{1}{\sqrt{2\nu}}\exp\left(\nu t\right)\:\sqrt{\int_{-\infty}^{t}\left|F\left(s\right)\right|^{2}\,\exp\left(-2\nu s\right)\: ds},
\end{eqnarray*}
existence of the integral $\int_{-\infty}^{t}\exp\left(-\left(t-s\right)\sqrt{M_{0}^{-1}}A\sqrt{M_{0}^{-1}}\right)F\left(s\right)\: ds$
and continuity of $t\mapsto\int_{-\infty}^{t}\exp\left(-\left(t-s\right)\sqrt{M_{0}^{-1}}A\sqrt{M_{0}^{-1}}\right)F\left(s\right)\: ds$
is clear. Since $F=0$ on $\mathbb{R}_{<0}$ we have in particular
that
\[
\int_{-\infty}^{0}\exp\left(-\left(t-s\right)\sqrt{M_{0}^{-1}}A\sqrt{M_{0}^{-1}}\right)F\left(s\right)\: ds=0.
\]
 Since the first term of (\ref{eq:ireg}) satisfies property (\ref{eq:idata})
we have indeed
\[
U\left(0+\right)=\sqrt{M_{0}^{-1}}\sqrt{M_{0}^{-1}}W_{0}=M_{0}^{-1}W_{0}.
\]
\end{proof}
\begin{rem}
In the case $A=0$ even stronger regularity follows. Indeed, the solution
is then such that
\begin{equation}
U-\chi_{_{\mathbb{R}_{\geq0}}}\otimes M_{0}^{-1}W_{0}\in H_{\nu,1}\left(\mathbb{R},H\right).\label{eq:ode-reg}
\end{equation}

\end{rem}

\section{Discussion of the Drude-Born-Fedorov Model}

We begin by analyzing the Drude-Born-Fedorov material relation ($\epsilon,\mu\in\mathbb{R}_{>0}$,~
$\eta\in\mathbb{R}\setminus\left\{ 0\right\} $) 
\[
D=\epsilon\left(E+\eta\:\curl E\right),\: B=\mu\left(H+\eta\:\curl H\right)
\]
in the light of the above general theory by substituting Maxwell's
equations back into the Drude-Born-Fedorov relation (ignoring possible
source terms) to obtain a modified material relation of the form 

\[
D=\epsilon\left(E\lyxmathsym{\textminus}\eta\:\partial_{0}B\right),\: B=\mu\left(H+\eta\:\partial_{0}D\right).
\]
This is a modified Condon model, \cite[Condon 1937]{Condon37}, compared
to which in the right-hand sides $B$ and $D$ are replaced by $H$
and $E$, respectively. 

Although similar in concept to the Condon model, where replacing a
convolution term by a two-term Taylor approximation leads to no reasonable
results, here this type of {}``approximation'' can actually be justified,
see \cite[Frantzeskakis, Ioannidis, Roach, Stratis,  Yannacopoulos (2003)]{1052.78002}.
With a slight reformulation we have Maxwell's equations 
\[
\partial_{0}\left(\begin{array}{c}
\frac{1}{\sqrt{\epsilon}}D\\
\frac{1}{\sqrt{\mu}}B
\end{array}\right)+\left(\begin{array}{cc}
0 & -\curl\\
\curl & 0
\end{array}\right)\left(\begin{array}{c}
\frac{1}{\sqrt{\mu}}E\\
\frac{1}{\sqrt{\epsilon}}H
\end{array}\right)=\left(\begin{array}{c}
-\frac{1}{\sqrt{\epsilon}}J\\
0
\end{array}\right)
\]
and the Drude-Born-Fedorov material relation assumes the formal shape

\[
\left(\begin{array}{cc}
1 & \eta\:\sqrt{\epsilon\mu}\partial_{0}\\
-\eta\:\sqrt{\epsilon\mu}\partial_{0} & 1
\end{array}\right)\left(\begin{array}{c}
\frac{1}{\sqrt{\epsilon}}D\\
\frac{1}{\sqrt{\mu}}B
\end{array}\right)=\left(\begin{array}{c}
\frac{1}{\sqrt{\mu}}E\\
\frac{1}{\sqrt{\epsilon}}H
\end{array}\right),
\]
 or rather
\begin{eqnarray*}
\left(\begin{array}{c}
\frac{1}{\sqrt{\epsilon}}D\\
\frac{1}{\sqrt{\mu}}B
\end{array}\right) & = & \left(\begin{array}{cc}
1 & \eta\:\sqrt{\epsilon\mu}\partial_{0}\\
-\eta\:\sqrt{\epsilon\mu}\partial_{0} & 1
\end{array}\right)^{-1}\left(\begin{array}{c}
\frac{1}{\sqrt{\mu}}E\\
\frac{1}{\sqrt{\epsilon}}H
\end{array}\right)\\
 & = & \left(1+\eta\:^{2}\epsilon\mu\partial_{0}^{2}\right)^{-1}\left(\begin{array}{cc}
1 & -\eta\:\sqrt{\epsilon\mu}\partial_{0}\\
\eta\:\sqrt{\epsilon\mu}\partial_{0} & 1
\end{array}\right)\left(\begin{array}{c}
\frac{1}{\sqrt{\mu}}E\\
\frac{1}{\sqrt{\epsilon}}H
\end{array}\right).
\end{eqnarray*}
In the language of the above theory
\begin{eqnarray*}
M\left(\partial_{0}^{-1}\right) & = & \left(1+\eta\:^{2}\epsilon\mu\partial_{0}^{2}\right)^{-1}\left(\begin{array}{cc}
1 & -\eta\:\sqrt{\epsilon\mu}\partial_{0}\\
\eta\:\sqrt{\epsilon\mu}\partial_{0} & 1
\end{array}\right)\\
 & = & \left(1+\left(\frac{1}{\eta\:\sqrt{\epsilon\mu}}\partial_{0}^{-1}\right)^{2}\right)^{-1}\left(\begin{array}{cc}
\left(\frac{1}{\eta\:\sqrt{\epsilon\mu}}\partial_{0}^{-1}\right)^{2} & -\left(\frac{1}{\eta\:\sqrt{\epsilon\mu}}\partial_{0}^{-1}\right)\\
\left(\frac{1}{\eta\:\sqrt{\epsilon\mu}}\partial_{0}^{-1}\right) & \left(\frac{1}{\eta\:\sqrt{\epsilon\mu}}\partial_{0}^{-1}\right)^{2}
\end{array}\right)\\
 & = & \left(\begin{array}{cc}
0 & 0\\
0 & 0
\end{array}\right)+\partial_{0}^{-1}\left(\begin{array}{cc}
0 & -\frac{1}{\eta\:\sqrt{\epsilon\mu}}\\
\frac{1}{\eta\:\sqrt{\epsilon\mu}} & 0
\end{array}\right)+\left(\frac{1}{\eta\:\sqrt{\epsilon\mu}}\partial_{0}^{-1}\right)^{2}\left(1+\left(\frac{1}{\eta\:\sqrt{\epsilon\mu}}\partial_{0}^{-1}\right)^{2}\right)^{-1}\\
 & = & \left(\begin{array}{cc}
0 & 0\\
0 & 0
\end{array}\right)+\partial_{0}^{-1}\left(\begin{array}{cc}
0 & -\frac{1}{\eta\:\sqrt{\epsilon\mu}}\\
\frac{1}{\eta\:\sqrt{\epsilon\mu}} & 0
\end{array}\right)+O\left(\partial_{0}^{-2}\right).
\end{eqnarray*}
Note
\[
\Re\left(\begin{array}{cc}
0 & -\frac{1}{\eta\:\sqrt{\epsilon\mu}}\\
\frac{1}{\eta\:\sqrt{\epsilon\mu}} & 0
\end{array}\right)=0
\]
which shows that the resulting equation is \emph{not} covered by the
above theory directly and actually is not a differential equation
in time at all:

\begin{eqnarray*}
 &  & \left(\frac{1}{\eta\:\sqrt{\epsilon\mu}}\partial_{0}^{-1}\left(\begin{array}{cc}
0 & -1\\
1 & 0
\end{array}\right)+\left(\frac{1}{\eta\:\sqrt{\epsilon\mu}}\partial_{0}^{-1}\right)^{2}\left(1+\left(\frac{1}{\eta\:\sqrt{\epsilon\mu}}\partial_{0}^{-1}\right)^{2}\right)^{-1}+\left(\begin{array}{cc}
0 & -\curl\\
\curl & 0
\end{array}\right)\right)\left(\begin{array}{c}
\frac{1}{\sqrt{\mu}}E\\
\frac{1}{\sqrt{\epsilon}}H
\end{array}\right)=\\
 &  & =\left(\begin{array}{c}
-\frac{1}{\sqrt{\epsilon}}J\\
\frac{1}{\sqrt{\mu}}\:0
\end{array}\right).
\end{eqnarray*}
Moreover, this whole consideration is assuming that we consider $\curl$
to be a self-adjoint realization in $L^{2}\left(\Omega\right)$, which
can only be achieved in very special cases, such as $\Omega=\mathbb{R}^{3}\setminus N$,
where $N$ is a set of capacity zero. (e.g. in $\Omega=\mathbb{R}^{3}$). 

However, for media occupying an arbitrary open subset $\Omega$ of
$\mathbb{R}^{3}$ we fortunately have a natural choice of boundary
condition, which turns $\curl$ with a corresponding choice of domain
into a selfadjoint operator. To formulate this condition properly
we need to introduce $\interior{\curl}$ as the closure in $L^{2}\left(\Omega\right)$
of $\curl$ restricted to $\interior{C}_{\infty}\left(\Omega\right)$
vector fields (we do \emph{not} indicate the number of components).
Then $\curl$ is properly defined as the adjoint of $\interior{\curl}$:
\begin{eqnarray*}
\curl & := & \left(\interior{\curl}\right)^{*}.
\end{eqnarray*}
Similarly, $\interior{\dive}$ is defined as the closure in $L^{2}\left(\Omega\right)$
of $\dive$ restricted to $\interior{C}_{\infty}\left(\Omega\right)$
vector fields. Then
\begin{eqnarray*}
\grad & := & -\left(\interior{\dive}\right)^{*}
\end{eqnarray*}
is the usual weak derivative in $L^{2}\left(\Omega\right)$, , see
\cite[$\pi $k 1998]{0926.35104} for the conceptual details. Containment
of a field $E$ in $D\left(\interior{\curl}\right)$ is the proper
weak generalization of the classical boundary condition {}``$n\times E=0\:\mbox{on }\partial\Omega$'',
whereas $E\in D\left(\interior{\dive}\right)$ generalized the classical
boundary condition {}``$n\cdot E=0\:\mbox{on }\partial\Omega$''.
It is important here to keep in mind that no regularity assumptions
on the boundary and no trace results are needed for these generalized
constructions.

A suitable boundary condition for the Drude-Born-Fedorov model can
now be stated in terms of the range of $\interior{\curl}$, $R\left(\interior{\curl}\right)$,
and of the domain of $\interior{\dive}$, $D\left(\interior{\dive}\right)$. 

We require 
\begin{equation}
\curl E\in\overline{R\left(\interior{\curl}\right)}\label{eq:bc0}
\end{equation}
or equivalently
\begin{equation}
\curl E\in D\left(\interior{\dive}\right)\label{eq:bca}
\end{equation}
and
\begin{equation}
\curl E\perp\mathcal{H}_{N},\label{eq:bcb}
\end{equation}
where $\mathcal{H}_{N}$ denotes the set of harmonic Neumann fields
\[
\mathcal{H}_{N}=\left\{ E\in D\left(\interior{\dive}\right)|\,\dive E=0,\curl E=0\right\} .
\]
Condition (\ref{eq:bca}) generalizes the classical boundary condition
{}``$n\cdot\curl E=0$ on $\partial\Omega$'' to non-smooth boundaries
and data%
\footnote{Note that assuming data in $D\left(\interior{\dive}\right)$ boundary
condition (\ref{eq:bca}) is induced by the requirement that $E\in D\left(\interior{\dive}\right)$.
Moreover, in the simply connected case $\mathcal{H}_{N}=\left\{ 0\right\} $
so that (\ref{eq:bc0}) reduces to (\ref{eq:bca}). These observations
are the link to the set-up utilized in \cite{1177.35042}.%
}. We shall denote the operator $\curl$ subject to boundary conditions
(\ref{eq:bca}), (\ref{eq:bcb}) by $\overset{\diamond}{\curl}$.
Under fairly general assumptions it can be shown that $\overset{\diamond}{\curl}$
is actually selfadjoint and even has -- apart from $0$ -- only discrete
spectrum $\sigma_{d}\left(\overset{\diamond}{\curl}\right)$, 
\begin{equation}
\sigma_{p}\left(\overset{\diamond}{\curl}\right)\setminus\left\{ 0\right\} =\sigma_{d}\left(\overset{\diamond}{\curl}\right).\label{eq:disc}
\end{equation}
The most general result, \cite[Filonov 2000]{Filonov00}, merely requires
$\Omega$ to be an open set with bounded measure to obtain the same
properties for $\overset{\diamond}{\curl}$. Exterior domains, i.e.
open sets with compact complement, also support selfadjointness of
$\overset{\diamond}{\curl}$, \cite[$\pi $k 1998]{0935.35029}, if
the boundary satisfies a local compact embedding property, see \cite{Picard2001}.
Indeed, although in the exterior domain case the spectrum of $\overset{\diamond}{\curl}$
is not purely point spectrum anymore, we still maintain (\ref{eq:disc}). 

In order to avoid technicalities and to immunize the results presented
here against possible future improvements, we make the crucial property
of selfadjointness our core assumption for the analysis of the Drude-Born-Fedorov
model:

\noindent \begin{center}
\begin{minipage}[t]{0.8\columnwidth}%
\textbf{Hypothesis} $\mathbf{\Omega}$: We assume that $\Omega$ is
an open subset of $\mathbb{R}^{3}$ such that 
\begin{itemize}
\item $\overset{\diamond}{\curl}$ is selfadjoint.\end{itemize}
\end{minipage}
\par\end{center}

Under this general assumption we shall re-inspect the Drude-Born-Fedorov
model from another perspective. We first note that 

\[
D=\epsilon\left(E+\eta\:\curl E\right)=\left(1+\eta\:\curl\right)\epsilon E,\: B=\mu\left(H+\eta\:\curl H\right)=\left(1+\eta\:\curl\right)\mu H.
\]
So, imposing our general \textbf{Hypothesis} $\mathbf{\Omega}$ the
material relation takes on the form 
\[
\left(\begin{array}{c}
D\\
B
\end{array}\right)=\left(1+\eta\:\overset{\diamond}{\curl}\right)\left(\begin{array}{cc}
\epsilon & 0\\
0 & \mu
\end{array}\right)\left(\begin{array}{c}
E\\
H
\end{array}\right).
\]
The initial value problem for Maxwell's equation now reads formally
\[
\left(1+\eta\:\mathbf{\overset{\diamond}{\curl}}\right)\partial_{0}\left(\begin{array}{cc}
\epsilon & 0\\
0 & \mu
\end{array}\right)\left(\begin{array}{c}
E\\
H
\end{array}\right)+\left(\begin{array}{cc}
0 & -1\\
1 & 0
\end{array}\right)\overset{\diamond}{\curl}\left(\begin{array}{c}
E\\
H
\end{array}\right)=\mathcal{J}+\delta\otimes W_{0}.
\]
Now let $P_{\eta}$ denote the orthogonal projector onto $\overline{R\left(1+\eta\:\mathbf{\overset{\diamond}{\curl}}\right)}$.
Consequently, $\left(1-P_{\eta}\right)$ is the orthogonal projector
onto $N\left(1+\eta\:\mathbf{\overset{\diamond}{\curl}}\right)$ and
so
\begin{eqnarray}
\left(1-P_{\eta}\right)\overset{\diamond}{\curl} & = & \eta^{-1}\left(1-P_{\eta}\right)\left(\eta\:\overset{\diamond}{\curl}+1-1\right)\nonumber \\
 & = & -\eta^{-1}\left(1-P_{\eta}\right)\label{eq:p}
\end{eqnarray}
\begin{eqnarray*}
\overset{\diamond}{\curl}\left(1-P_{\eta}\right) & = & \eta^{-1}\left(\eta\overset{\diamond}{\curl}+1-1\right)\left(1-P_{\eta}\right)\\
 & = & \eta^{-1}\left(\eta\overset{\diamond}{\curl}+1\right)\left(1-P_{\eta}\right)-\eta^{-1}\left(1-P_{\eta}\right)\\
 & = & -\eta^{-1}\left(1-P_{\eta}\right).
\end{eqnarray*}
Thus, $P_{\eta}$ and $\overset{\diamond}{\curl}$ commute and we
have 
\begin{eqnarray*}
\left(1+\eta\:\mathbf{\overset{\diamond}{\curl}}\right)\partial_{0}P_{\eta}\left(\begin{array}{cc}
\epsilon & 0\\
0 & \mu
\end{array}\right)\left(\begin{array}{c}
E\\
H
\end{array}\right)+\\
+\left(\begin{array}{cc}
0 & -1\\
1 & 0
\end{array}\right)P_{\eta}\overset{\diamond}{\curl}\left(\begin{array}{c}
E\\
H
\end{array}\right) & = & P_{\eta}\mathcal{J}+\delta\otimes P_{\eta}W_{0}\\
\left(1+\eta\:\mathbf{\overset{\diamond}{\curl}}\right)\partial_{0}P_{\eta}\left(\begin{array}{cc}
\epsilon & 0\\
0 & \mu
\end{array}\right)P_{\eta}\left(\begin{array}{c}
E\\
H
\end{array}\right)+\\
+\left(1+\eta\:\mathbf{\overset{\diamond}{\curl}}\right)\partial_{0}P_{\eta}\left(\begin{array}{cc}
\epsilon & 0\\
0 & \mu
\end{array}\right)\left(1-P_{\eta}\right)\left(\begin{array}{c}
E\\
H
\end{array}\right)+\\
+\left(\begin{array}{cc}
0 & -1\\
1 & 0
\end{array}\right)P_{\eta}\overset{\diamond}{\curl}\left(\begin{array}{c}
E\\
H
\end{array}\right) & = & P_{\eta}\mathcal{J}+\delta\otimes P_{\eta}W_{0}\\
\left(\begin{array}{cc}
0 & -1\\
1 & 0
\end{array}\right)\left(1-P_{\eta}\right)\overset{\diamond}{\curl}\left(\begin{array}{c}
E\\
H
\end{array}\right) & = & \left(1-P_{\eta}\right)\mathcal{J}+\delta\otimes\left(1-P_{\eta}\right)W_{0}\\
-\eta^{-1}\left(\begin{array}{cc}
0 & -1\\
1 & 0
\end{array}\right)\left(1-P_{\eta}\right)\left(\begin{array}{c}
E\\
H
\end{array}\right) & = & 0
\end{eqnarray*}
and so
\begin{eqnarray*}
\left(1+\eta\:\pi_{\eta}\overset{\diamond}{\curl}\pi_{\eta}^{*}\right)\partial_{0}\pi_{\eta}\left(\begin{array}{cc}
\epsilon & 0\\
0 & \mu
\end{array}\right)\pi_{\eta}^{*}\pi_{\eta}\left(\begin{array}{c}
E\\
H
\end{array}\right)+\\
+\left(\begin{array}{cc}
0 & -1\\
1 & 0
\end{array}\right)\pi_{\eta}\overset{\diamond}{\curl}\pi_{\eta}^{*}\pi_{\eta}\left(\begin{array}{c}
E\\
H
\end{array}\right) & = & \pi_{\eta}\mathcal{J}+\delta\otimes\pi_{\eta}W_{0}
\end{eqnarray*}
 
\begin{eqnarray*}
\partial_{0}\left(\pi_{\eta}\left(\begin{array}{cc}
\epsilon & 0\\
0 & \mu
\end{array}\right)\pi_{\eta}^{*}\right)\pi_{\eta}\left(\begin{array}{c}
E\\
H
\end{array}\right)+\\
+\left(1+\eta\:\pi_{\eta}\overset{\diamond}{\curl}\pi_{\eta}^{*}\right)^{-1}\pi_{\eta}\overset{\diamond}{\curl}\pi_{\eta}^{*}\left(\begin{array}{cc}
0 & -1\\
1 & 0
\end{array}\right)\pi_{\eta}\left(\begin{array}{c}
E\\
H
\end{array}\right) & = & \pi_{\eta}\mathcal{J}+\delta\otimes\pi_{\eta}W_{0}
\end{eqnarray*}
It is
\[
P_{\eta}\overset{\diamond}{\curl}\subseteq\overset{\diamond}{\curl}P_{\eta}
\]
or
\[
\left(\partial_{0}\left(\begin{array}{cc}
\epsilon & 0\\
0 & \mu
\end{array}\right)+\left(1+\eta\:\overset{\diamond}{\curl}\right)^{-1}\overset{\diamond}{\curl}\left(\begin{array}{cc}
0 & -1\\
1 & 0
\end{array}\right)\left(\begin{array}{c}
E\\
H
\end{array}\right)\right)=\left(1+\eta\:\overset{\diamond}{\curl}\right)^{-1}(\mathcal{J}+\delta\otimes W_{0}).
\]

Discussing this evolution equation we can finally make rigorous sense
of our above arguments. 

The solution theory of the original Drude-Born-Fedorov model can now
be formulated within the framework of the Sobolev lattice

\[
\left(H_{\nu,k}\left(\mathbb{R},\: H_{s}\left(\overset{\diamond}{\curl}+\i\right)\right)\right)_{k,s\in\mathbb{Z}}.
\]

\begin{thm}
\label{DBF0}Under \textbf{Hypothesis} $\mathbf{\Omega}$ and assuming
that 
\begin{equation}
-\frac{1}{\eta}\in\left(\rho\left(\overset{\diamond}{\curl}\right)\cap\mathbb{R}\right)\cup\left\{ \lambda|\,\lambda\:\mbox{ is isolated in }\sigma\left(\overset{\diamond}{\curl}\right)\right\} \label{eq:spec}
\end{equation}
we have that for 
\begin{equation}
\mathcal{J}\in H_{\nu,0}\left(\mathbb{R},R\left(1+\eta\:\overset{\diamond}{\curl}\right)\right),\:\mathcal{J}=0\mbox{ on }\mathbb{R}_{<0},\label{eq:range1}
\end{equation}
 and 
\begin{equation}
W_{0}\in R\left(1+\eta\:\overset{\diamond}{\curl}\right)\label{eq:range2}
\end{equation}
 the Drude-Born-Fedorov model 
\begin{equation}
\partial_{0}\left(\begin{array}{c}
D\\
B
\end{array}\right)+\left(\begin{array}{cc}
0 & -\overset{\diamond}{\curl}\\
\overset{\diamond}{\curl} & 0
\end{array}\right)\left(\begin{array}{c}
E\\
H
\end{array}\right)=\mathcal{J}+\delta\otimes W_{0}\label{eq:DBF-evo}
\end{equation}
with
\begin{equation}
\left(\begin{array}{c}
D\\
B
\end{array}\right)=\left(1+\eta\:\overset{\diamond}{\curl}\right)\left(\begin{array}{cc}
\epsilon & 0\\
0 & \mu
\end{array}\right)\left(\begin{array}{c}
E\\
H
\end{array}\right)\label{eq:material-DBF}
\end{equation}
 has a unique solution $\left(\begin{array}{c}
E\\
H
\end{array}\right)\in H_{\nu,0}\left(\mathbb{R},L^{2}\left(\Omega\right)\right)$ satisfying
\[
\left(\begin{array}{c}
D\\
B
\end{array}\right)\left(0+\right)=W_{0}
\]
in $H_{-1}\left(\overset{\diamond}{\curl}+\i\right)$. The solution
depends continuously and causal on the data.\end{thm}
\begin{proof}
We have that 
\begin{eqnarray*}
\overset{\diamond}{\curl}_{\eta}:D\left(\overset{\diamond}{\curl}\right)\cap\overline{R\left(1+\eta\:\overset{\diamond}{\curl}\right)}\subseteq\overline{R\left(1+\eta\:\overset{\diamond}{\curl}\right)} & \to & R\left(1+\eta\:\overset{\diamond}{\curl}\right)\subseteq\overline{R\left(1+\eta\:\overset{\diamond}{\curl}\right)}\\
\phi & \mapsto & \overset{\diamond}{\curl}\phi
\end{eqnarray*}
is a selfadjoint operator with a well-defined continuous inverse
\begin{eqnarray*}
\left(1+\eta\:\overset{\diamond}{\curl}_{\eta}\right)^{-1}:\overline{R\left(1+\eta\:\overset{\diamond}{\curl}\right)} & \to & R\left(1+\eta\:\overset{\diamond}{\curl}\right)\subseteq\overline{R\left(1+\eta\:\overset{\diamond}{\curl}\right)}.
\end{eqnarray*}
It is

\[
\pi_{\eta}\overset{\diamond}{\curl}\pi_{\eta}^{*}=\overset{\diamond}{\curl}_{\eta}
\]
where $\pi_{\eta}$ denotes the canonical projection onto the subspace
$\overline{R\left(1+\eta\:\overset{\diamond}{\curl}\right)}$ of $L^{2}\left(\Omega\right)$.
Then $\pi_{\eta}^{*}$ is the identity embedding of $\overline{R\left(1+\eta\:\overset{\diamond}{\curl}\right)}$
into $L^{2}\left(\Omega\right)$.

We consider initially the solution of the evolution equation
\begin{eqnarray}
 &  & \left(\partial_{0}\pi_{\eta}\left(\begin{array}{cc}
\epsilon & 0\\
0 & \mu
\end{array}\right)\pi_{\eta}^{*}+\left(1+\eta\:\overset{\diamond}{\curl}_{\eta}\right)^{-1}\overset{\diamond}{\curl}_{\eta}\left(\begin{array}{cc}
0 & -1\\
1 & 0
\end{array}\right)\right)\left(\begin{array}{c}
e\\
h
\end{array}\right)=\label{eq:ode}\\
 &  & =\left(1+\eta\:\overset{\diamond}{\curl}_{\eta}\right)^{-1}\mathcal{J}+\nonumber \\
 &  & +\delta\otimes\left(1+\eta\:\overset{\diamond}{\curl}_{\eta}\right)^{-1}W_{0}\nonumber 
\end{eqnarray}
and observe that the resulting system is now an evolutionary system
in the sense of Section \ref{sec:Space-Time-Evolution-Equations}
with
\begin{eqnarray}
M\left(\partial_{0}^{-1}\right) & = & \pi_{\eta}\left(\begin{array}{cc}
\epsilon & 0\\
0 & \mu
\end{array}\right)\pi_{\eta}^{*}+\partial_{0}^{-1}C,\label{eq:mod-mat}
\end{eqnarray}
$A=0$ and 
\begin{eqnarray*}
C & = & \left(1+\eta\:\overset{\diamond}{\curl}_{\eta}\right)^{-1}\overset{\diamond}{\curl}_{\eta}\left(\begin{array}{cc}
0 & -1\\
1 & 0
\end{array}\right)\\
 & = & \left(\begin{array}{cc}
0 & -\left(1+\eta\:\overset{\diamond}{\curl}_{\eta}\right)^{-1}\overset{\diamond}{\curl}_{\eta}\\
\left(1+\eta\:\overset{\diamond}{\curl}_{\eta}\right)^{-1}\overset{\diamond}{\curl}_{\eta} & 0
\end{array}\right)
\end{eqnarray*}
continuous on $\overline{R\left(1+\eta\:\overset{\diamond}{\curl}\right)}$.
\begin{eqnarray*}
\left(1+\eta\:\overset{\diamond}{\curl}_{\eta}\right)^{-1}\overset{\diamond}{\curl}_{\eta} & = & \eta^{-1}\left(1+\eta\:\overset{\diamond}{\curl}_{\eta}\right)^{-1}\left(\eta\overset{\diamond}{\curl}_{\eta}+1\right)-\eta^{-1}\left(1+\eta\:\overset{\diamond}{\curl}_{\eta}\right)^{-1}\\
 & = & \eta^{-1}-\eta^{-1}\left(1+\eta\:\overset{\diamond}{\curl}_{\eta}\right)^{-1}
\end{eqnarray*}
From the abstract theory we now find existence, uniqueness of a solution
in the space $H_{\nu,0}\left(\mathbb{R},\overline{R\left(1+\eta\:\overset{\diamond}{\curl}\right)}\right)$
and continuous (causal) dependence on the right-hand side. Moreover, 

\[
\pi_{\eta}\left(\begin{array}{cc}
\epsilon & 0\\
0 & \mu
\end{array}\right)\pi_{\eta}^{*}\left(\begin{array}{c}
e\\
h
\end{array}\right)-\chi_{_{\mathbb{R}_{>0}}}\otimes\left(1+\eta\:\overset{\diamond}{\curl}_{\eta}\right)^{-1}W_{0}\in H_{\nu,1}\left(\mathbb{R},\overline{R\left(1+\eta\:\overset{\diamond}{\curl}\right)}\right).
\]

With 
\begin{eqnarray*}
\left(\begin{array}{c}
E\\
H
\end{array}\right) & := & \pi_{\eta}^{*}\left(\begin{array}{c}
e\\
h
\end{array}\right)=\pi_{\eta}^{*}\pi_{\eta}\pi_{\eta}^{*}\left(\begin{array}{c}
e\\
h
\end{array}\right)
\end{eqnarray*}
and using that $P_{\eta}=\pi_{\eta}^{*}\pi_{\eta}$ commutes with
$\left(\begin{array}{cc}
\epsilon & 0\\
0 & \mu
\end{array}\right)$ we get 
\[
\left(\begin{array}{cc}
\epsilon & 0\\
0 & \mu
\end{array}\right)\left(\begin{array}{c}
E\\
H
\end{array}\right)-\chi_{_{\mathbb{R}_{>0}}}\otimes\pi_{\eta}^{*}\left(1+\eta\:\overset{\diamond}{\curl}_{\eta}\right)^{-1}W_{0}\in H_{\nu,1}\left(\mathbb{R},L^{2}\left(\Omega\right)\right).
\]

Applying $\left(1+\eta\:\overset{\diamond}{\curl}\right)$ now yields
\begin{eqnarray*}
 &  & \left(1+\eta\:\overset{\diamond}{\curl}\right)\left(\begin{array}{cc}
\epsilon & 0\\
0 & \mu
\end{array}\right)\left(\begin{array}{c}
E\\
H
\end{array}\right)+\\
 &  & -\chi_{_{\mathbb{R}_{>0}}}\otimes\left(1+\eta\:\overset{\diamond}{\curl}\right)\pi_{\eta}^{*}\left(1+\eta\:\overset{\diamond}{\curl}_{\eta}\right)^{-1}W_{0}\in H_{\nu,1}\left(\mathbb{R},H_{-1}\left(\overset{\diamond}{\curl}+\i\right)\right)
\end{eqnarray*}
Clearly, 
\begin{eqnarray*}
\left(1+\eta\:\overset{\diamond}{\curl}\right)\pi_{\eta}^{*}\left(1+\eta\:\overset{\diamond}{\curl}_{\eta}\right)^{-1}W_{0} & = & \pi_{\eta}\left(1+\eta\:\overset{\diamond}{\curl}\right)\pi_{\eta}^{*}\left(1+\eta\:\overset{\diamond}{\curl}_{\eta}\right)^{-1}W_{0}\\
 & = & \left(1+\eta\:\overset{\diamond}{\curl}_{\eta}\right)\left(1+\eta\:\overset{\diamond}{\curl}_{\eta}\right)^{-1}W_{0}\\
 & = & W_{0}
\end{eqnarray*}
and so
\[
\left(1+\eta\:\overset{\diamond}{\curl}\right)\left(\begin{array}{cc}
\epsilon & 0\\
0 & \mu
\end{array}\right)\left(\begin{array}{c}
E\\
H
\end{array}\right)-\chi_{_{\mathbb{R}_{>0}}}\otimes W_{0}\in H_{\nu,1}\left(\mathbb{R},H_{-1}\left(\overset{\diamond}{\curl}+\i\right)\right).
\]
The latter shows by causality and a temporal Sobolev embedding argument
that 
\[
t\mapsto\left(\left(1+\eta\:\overset{\diamond}{\curl}\right)\left(\begin{array}{cc}
\epsilon & 0\\
0 & \mu
\end{array}\right)\left(\begin{array}{c}
E\\
H
\end{array}\right)-\chi_{_{\mathbb{R}_{>0}}}\otimes W_{0}\right)\left(t\right)
\]
 is continuous on $\mathbb{R}$ and in particular 
\[
\left(\left(1+\eta\:\overset{\diamond}{\curl}\right)\left(\begin{array}{cc}
\epsilon & 0\\
0 & \mu
\end{array}\right)\left(\begin{array}{c}
E\\
H
\end{array}\right)-\chi_{_{\mathbb{R}_{>0}}}\otimes W_{0}\right)\left(0\right)=0.
\]
The latter implies
\[
\left(\begin{array}{c}
D\\
B
\end{array}\right)\left(0+\right)=W_{0}.
\]
Applying $\left(1+\eta\:\overset{\diamond}{\curl}\right)\pi_{\eta}^{*}$
to equation (\ref{eq:ode}) yields similarly 
\[
\partial_{0}\left(\begin{array}{c}
D\\
B
\end{array}\right)+\overset{\diamond}{\curl}\left(\begin{array}{cc}
0 & -1\\
1 & 0
\end{array}\right)\left(\begin{array}{c}
E\\
H
\end{array}\right)=\mathcal{J}+\delta\otimes W_{0},
\]
which is the equation of the Drude-Born-Fedorov model. Note that this
equation now holds in $H_{\nu,-1}(\mathbb{R},H_{-1}(\overset{\diamond}{\curl}+\i)).$
Causality and continuity estimates follow from the general theory
applied to (\ref{eq:ode}). 

Conversely, if $\left(\begin{array}{c}
E\\
H
\end{array}\right)$ solves (\ref{eq:DBF-evo}) with (\ref{eq:material-DBF}) with $\mathcal{J}=0$
and $W_{0}=0$, then 
\begin{equation}
\left(\partial_{0}\left(1+\eta\:\overset{\diamond}{\curl}\right)\left(\begin{array}{cc}
\epsilon & 0\\
0 & \mu
\end{array}\right)+\overset{\diamond}{\curl}\left(\begin{array}{cc}
0 & -1\\
1 & 0
\end{array}\right)\right)\left(\begin{array}{c}
E\\
H
\end{array}\right)=0\label{eq:reg-0-DBF}
\end{equation}
and so

\begin{equation}
\left(1+\eta\:\overset{\diamond}{\curl}\right)\left(\partial_{0}\left(\begin{array}{cc}
\epsilon & 0\\
0 & \mu
\end{array}\right)+\eta^{-1}\left(\begin{array}{cc}
0 & -1\\
1 & 0
\end{array}\right)\right)\left(\begin{array}{c}
E\\
H
\end{array}\right)=\eta^{-1}\left(\begin{array}{cc}
0 & -1\\
1 & 0
\end{array}\right)\left(\begin{array}{c}
E\\
H
\end{array}\right).\label{eq:uni}
\end{equation}
Thus, we have that
\[
\eta^{-1}\left(\begin{array}{cc}
0 & -1\\
1 & 0
\end{array}\right)\left(\begin{array}{c}
E\\
H
\end{array}\right)\in R\left(1+\eta\:\overset{\diamond}{\curl}\right).
\]
This observation implies
\[
\left(\partial_{0}P_{\eta}\left(\begin{array}{cc}
\epsilon & 0\\
0 & \mu
\end{array}\right)+P_{\eta}\eta^{-1}\left(\begin{array}{cc}
0 & -1\\
1 & 0
\end{array}\right)\right)\left(\begin{array}{c}
E\\
H
\end{array}\right)=\left(1+\eta\:\overset{\diamond}{\curl}_{\eta}\right)^{-1}\eta^{-1}\left(\begin{array}{cc}
0 & -1\\
1 & 0
\end{array}\right)\left(\begin{array}{c}
E\\
H
\end{array}\right)
\]
or
\[
\left(\partial_{0}P_{\eta}\left(\begin{array}{cc}
\epsilon & 0\\
0 & \mu
\end{array}\right)+\overset{\diamond}{\curl}_{\eta}\left(1+\eta\:\overset{\diamond}{\curl}_{\eta}\right)^{-1}\left(\begin{array}{cc}
0 & -1\\
1 & 0
\end{array}\right)\right)\left(\begin{array}{c}
E\\
H
\end{array}\right)=0.
\]
This implies 
\begin{eqnarray*}
0 & = & \Re\left\langle P_{\eta}\left(\begin{array}{c}
E\\
H
\end{array}\right)\Big|\left(\partial_{0}\left(\begin{array}{cc}
\epsilon & 0\\
0 & \mu
\end{array}\right)+\overset{\diamond}{\curl}_{\eta}\left(1+\eta\:\overset{\diamond}{\curl}_{\eta}\right)^{-1}\left(\begin{array}{cc}
0 & -1\\
1 & 0
\end{array}\right)\right)P_{\eta}\left(\begin{array}{c}
E\\
H
\end{array}\right)\right\rangle _{\nu,0,0}\\
 & = & \Re\left\langle P_{\eta}\left(\begin{array}{c}
E\\
H
\end{array}\right)\Big|\partial_{0}\left(\begin{array}{cc}
\epsilon & 0\\
0 & \mu
\end{array}\right)P_{\eta}\left(\begin{array}{c}
E\\
H
\end{array}\right)\right\rangle _{\nu,0,0}\\
 & = & \nu\left\langle P_{\eta}\left(\begin{array}{c}
E\\
H
\end{array}\right)\Big|\left(\begin{array}{cc}
\epsilon & 0\\
0 & \mu
\end{array}\right)P_{\eta}\left(\begin{array}{c}
E\\
H
\end{array}\right)\right\rangle _{\nu,0,0}
\end{eqnarray*}
and we see
\[
P_{\eta}\left(\begin{array}{c}
E\\
H
\end{array}\right)=0.
\]
On the other hand, we read off from (\ref{eq:uni}) that 
\begin{eqnarray*}
0 & = & \left(1-P_{\eta}\right)\left(\begin{array}{cc}
0 & -1\\
1 & 0
\end{array}\right)\left(\begin{array}{c}
E\\
H
\end{array}\right)\\
 & = & \left(\begin{array}{cc}
0 & -1\\
1 & 0
\end{array}\right)\left(1-P_{\eta}\right)\left(\begin{array}{c}
E\\
H
\end{array}\right)
\end{eqnarray*}
implying
\[
\left(1-P_{\eta}\right)\left(\begin{array}{c}
E\\
H
\end{array}\right)=0.
\]
This shows $\left(\begin{array}{c}
E\\
H
\end{array}\right)=0$, i.e. uniqueness.\end{proof}
\begin{rem}
\label{rem-reso}Since the known results on the spectrum of $\overset{\diamond}{\curl}$
show compact resolvent in the ortho-complement of the null space of
$\overset{\diamond}{\curl}$, we also have $\overline{R\left(1+\eta\:\overset{\diamond}{\curl}\right)}=R\left(1+\eta\:\overset{\diamond}{\curl}\right)$
in these cases. If we have $-\frac{1}{\eta}\in\rho\left(\overset{\diamond}{\curl}\right)$
the result becomes particularly simple. In this case we always have
$\overline{R\left(1+\eta\:\overset{\diamond}{\curl}\right)}=R\left(1+\eta\:\overset{\diamond}{\curl}\right)=L^{2}\left(\Omega\right)$,
which eliminates the possibly undesirable range conditions (\ref{eq:range1}),
(\ref{eq:range2}).

Note that the assumption of Theorem \ref{DBF0} may never hold (making
the claim of the theorem trivially correct), as e.g. in the exterior
domain case. However, for example in the case of $\Omega$ having
bounded measure, we have -- according to \cite{Filonov00} -- pure
point spectrum with no accumulation points and so the assumption of
Theorem \ref{DBF0} is non-trivial and says simply $-\frac{1}{\eta}\in\mathbb{R}$
or 
\[
\eta\in\mathbb{R}\setminus\left\{ 0\right\} .
\]
In the case $\eta=0$, although not covered by the theorem, the solution
theory is just the standard regular case.
\end{rem}

\section{A Note on Generalizations of the Drude-Born-Fedorov Model}

By virtue of the power of the theoretical framework the above solution
strategies can be conveniently generalized. We may straightforwardly
generalize the Drude-Born-Fedorov model (\ref{eq:material-DBF}) by
imposing instead 
\begin{equation}
\left(\begin{array}{c}
D\\
B
\end{array}\right)=\left(\kappa\left(\partial_{0}^{-1}\right)+\overset{\diamond}{\curl}\right)M_{*}\left(\partial_{0}^{-1}\right)\left(\begin{array}{c}
E\\
H
\end{array}\right)\label{eq:material-DBF-mod}
\end{equation}
where
\[
\kappa\left(\partial_{0}^{-1}\right)=\kappa_{0}+\partial_{0}^{-1}\kappa_{1}\left(\partial_{0}^{-1}\right)
\]
and 
\[
M_{*}\left(\partial_{0}^{-1}\right)=M_{*,0}+\partial_{0}^{-1}M_{*,1}\left(\partial_{0}^{-1}\right)
\]
with $\kappa_{0},\: M_{*,0}$ selfadjoint, continuous, strictly positive
definite and $\kappa,\: M_{*}$ bounded, analytic in $B_{\mathbb{C}}\left(r,r\right)$,
for some $r\in\mathbb{R}_{>0}$. It should be noted that convolution
integral type generalizations of the Drude-Born-Fedorov model have
been previously suggested in \cite{A2002285}, but discussed only
for the time-harmonic case. Also, the particular model based on homogenization
with $\kappa\left(\partial_{0}^{-1}\right)=\eta^{-1}$, $M_{*,1}=\eta\epsilon\: k\times$,
$M_{*,0}=\eta\left(\begin{array}{cc}
\epsilon & 0\\
0 & \mu
\end{array}\right)$, $k\in\mathbb{R}^{3}$, has been suggested in \cite{0022-3727-41-15-155412}.

Assuming 
\[
-1\in\rho\left(\kappa_{0}^{-1}\overset{\diamond}{\curl}\right)
\]
as a replacement for assumption (\ref{eq:spec}) we find that $\left(\kappa\left(\partial_{0}^{-1}\right)+\overset{\diamond}{\curl}\right)$
is boundedly invertible for all sufficiently large $\nu\in\mathbb{R}_{>\frac{1}{2r}}$.
Indeed, in this case 
\begin{eqnarray*}
 &  & \left(\kappa\left(\partial_{0}^{-1}\right)+\overset{\diamond}{\curl}\right)^{-1}=\\
 &  & =\left(\kappa\left(\partial_{0}^{-1}\right)+\overset{\diamond}{\curl}\right)^{-1}\\
 &  & =\left(\kappa_{0}-\partial_{0}^{-1}\kappa_{1}\left(\partial_{0}^{-1}\right)+\overset{\diamond}{\curl}\right)^{-1}\\
 &  & =\left(1-Q_{0}\left(\partial_{0}^{-1}\right)\right)^{-1}\left(\kappa_{0}+\overset{\diamond}{\curl}\right)^{-1}
\end{eqnarray*}
with 
\begin{eqnarray*}
Q_{0}\left(\partial_{0}^{-1}\right) & := & \left(\kappa_{0}+\overset{\diamond}{\curl}\right)^{-1}\partial_{0}^{-1}\kappa_{1}\left(\partial_{0}^{-1}\right)
\end{eqnarray*}
as an ad-hoc abbreviation. Moreover, 
\begin{eqnarray*}
 &  & \left(\kappa\left(\partial_{0}^{-1}\right)+\overset{\diamond}{\curl}\right)^{-1}\overset{\diamond}{\curl}=\\
 &  & =\left(1-Q_{0}\left(\partial_{0}^{-1}\right)\right)^{-1}\left(\kappa_{0}+\overset{\diamond}{\curl}\right)^{-1}\overset{\diamond}{\curl}
\end{eqnarray*}
shows that 
\[
C=\left(\kappa\left(\partial_{0}^{-1}\right)+\overset{\diamond}{\curl}\right)^{-1}\overset{\diamond}{\curl}
\]
is a bounded operator. So, with our assumption $-1\in\rho\left(\kappa_{0}\overset{\diamond}{\curl}\right)$,
as a by-product of the proof of Theorem \ref{DBF0}, the well-posedness
result extends also to this generalization. The modified material
law for the analogue of (\ref{eq:ode}) is 
\[
M\left(\partial_{0}^{-1}\right)=M_{*}\left(\partial_{0}^{-1}\right)+\partial_{0}^{-1}C.
\]
Note that the initial condition now assumes the form
\[
\left(\left(\kappa\left(\partial_{0}^{-1}\right)+\overset{\diamond}{\curl}\right)M_{*}\left(\partial_{0}^{-1}\right)\left(\begin{array}{c}
E\\
H
\end{array}\right)\right)\left(0+\right)=\left(\kappa_{0}+\overset{\diamond}{\curl}\right)M_{0,*}\left(\begin{array}{c}
E\\
H
\end{array}\right)\left(0+\right)=W_{0}.
\]
In summary, we obtain as a simple corollary to the proof of Theorem
\ref{DBF0} (taking Remark \ref{rem-reso} into account):
\begin{cor}
\label{DBF1}Let \textbf{Hypothesis} $\mathbf{\Omega}$ hold and let
\[
\kappa\left(\partial_{0}^{-1}\right)=\kappa_{0}+\partial_{0}^{-1}\kappa_{1}\left(\partial_{0}^{-1}\right)
\]
and 
\[
M_{*}\left(\partial_{0}^{-1}\right)=M_{*,0}+\partial_{0}^{-1}M_{*,1}\left(\partial_{0}^{-1}\right)
\]
with $\kappa_{0},\: M_{*,0}$ selfadjoint, continuous, strictly positive
definite and $\kappa,\: M_{*}$ bounded, analytic in $B_{\mathbb{C}}\left(r,r\right)$,
for some $r\in\mathbb{R}_{>0}$. Then we have that for every
\begin{equation}
-\frac{1}{\eta}\in\rho\left(\overset{\diamond}{\curl}\right)\cap\mathbb{R},
\end{equation}
 
\begin{equation}
\mathcal{J}\in H_{\nu,0}\left(\mathbb{R},L^{2}\left(\Omega\right)\right)
\end{equation}
 and 
\begin{equation}
W_{0}\in L^{2}\left(\Omega\right)
\end{equation}
 the Drude-Born-Fedorov model 
\begin{equation}
\partial_{0}\left(\begin{array}{c}
D\\
B
\end{array}\right)+\left(\begin{array}{cc}
0 & -\overset{\diamond}{\curl}\\
\overset{\diamond}{\curl} & 0
\end{array}\right)\left(\begin{array}{c}
E\\
H
\end{array}\right)=\mathcal{J}+\delta\otimes W_{0}
\end{equation}
with
\begin{equation}
\left(\begin{array}{c}
D\\
B
\end{array}\right)=\left(\kappa\left(\partial_{0}^{-1}\right)+\overset{\diamond}{\curl}\right)M_{*}\left(\partial_{0}^{-1}\right)\left(\begin{array}{c}
E\\
H
\end{array}\right)\label{eq:material-DBF-1}
\end{equation}
 has a unique solution $\left(\begin{array}{c}
E\\
H
\end{array}\right)\in H_{\nu,0}\left(\mathbb{R},L^{2}\left(\Omega\right)\right)$ satisfying
\[
\left(\begin{array}{c}
D\\
B
\end{array}\right)\left(0+\right)=W_{0}
\]
in $H_{-1}\left(\overset{\diamond}{\curl}+\i\right)$. Moreover, the
solution depends continuously and causal on the right-hand side data.
\end{cor}


\begin{thebibliography}{10}
\bibitem{1075.78005} Christodoulos Athanasiadis and Gary~F. Roach. \newblock {Time-dependent potential scattering in chiral media.} \newblock {\em J. Math. Anal. Appl.}, 310(1):1--15, 2005.
\bibitem{1126.78002} Patrick~jun. Ciarlet and Guillaume Legendre. \newblock {Well-posedness of the Drude-Born-Fedorov model for chiral media.} \newblock {\em Math. Models Methods Appl. Sci.}, 17(3):461--484, 2007.
\bibitem{pre05506047} Patrick~jun. Ciarlet and Guillaume Legendre. \newblock {Erratum: Well-posedness of the Drude-Born-Fedorov model for chiral   media.} \newblock {\em Math. Models Methods Appl. Sci.}, 19(1):173--174, 2009.
\bibitem{Condon37} E.~U. Condon. \newblock Theories of optical rotatory power. \newblock {\em Reviews of Modern Physics}, 9:0432--0457, 1937.
\bibitem{Filonov00} N.~Filonov. \newblock Spectral analysis of the selfadjoint operator $\mathrm{curl}$ in a   region of finite measure. \newblock {\em St. Petersburg Math. J.}, 11(6):1085--1095, 2000.
\bibitem{1052.78002} D.~J. Frantzeskakis, A.~Ioannidis, G.~F. Roach, I.~G. Stratis, and A.~N.   Yannacopoulos. \newblock {On the error of the optical response approximation in chiral media.} \newblock {\em Appl. Anal.}, 82(9):839--856, 2003.
\bibitem{Kato} T.~Kato. \newblock {\em Perturbation Theory for Linear Operators}, volume 132 of {\em   Grundlehren der mathematischen Wissenschaften}. \newblock Springer Verlag, Berlin, 2nd edition, 1976.
\bibitem{0866.00023} Akhlesh Lakhtakia. \newblock {\em {Beltrami Fields in Chiral Media.}} \newblock {Singapore: World Scientific., 568 p.}, 1994.
\bibitem{Leis:Buch:2} R.~Leis. \newblock {\em Initial boundary value problems in mathematical physics}. \newblock John Wiley \& Sons Ltd. and B.G. Teubner; Stuttgart, 1986.
\bibitem{1177.35042} K.B. Liaskos, I.G. Stratis, and A.N. Yannacopoulos. \newblock {A priori estimates for a singular limit approximation of the   constitutive laws for chiral media in the time domain.} \newblock {\em J. Math. Anal. Appl.}, 355(1):288--302, 2009.
\bibitem{LSTV} I.~V. Lindell, A.~H. Sihvola, S.~A. Tretyakov, and A.~J. Viitanen. \newblock {\em {Electromagnetic waves in chiral and bi-isotropic media}}. \newblock Artech House, Boston and London, 1994.
\bibitem{A2002285} A.~Morro. \newblock Modelling of optically active electromagnetic media. \newblock {\em Applied Mathematics Letters}, 15(3):285 -- 291, 2002.
\bibitem{0579.58030} R.~Picard. \newblock {On a structural observation in generalized electromagnetic theory.} \newblock {\em J. Math. Anal. Appl.}, 110:247--264, 1985.
\bibitem{0926.35104} R.~Picard. \newblock {On a selfadjoint realization of curl and some of its applications.} \newblock {\em Ric. Mat.}, 47(1):153--180, 1998.
\bibitem{0935.35029} R.~Picard. \newblock {On a selfadjoint realization of curl in exterior domains.} \newblock {\em Math. Z.}, 229(2):319--338, 1998.
\bibitem{Pi2009-1} R.~Picard. \newblock {A Structural Observation for Linear Material Laws in Classical   Mathematical Physics.} \newblock {\em {Math. Methods Appl. Sci.}}, 32(14):1768--1803, 2009.
\bibitem{2009-2} R.~Picard. \newblock {On a Class of Linear Material Laws in Classical Mathematical   Physics.} \newblock {\em Int. J. Pure Appl. Math.}, 50(2):283--288, 2009.
\bibitem{PIC_2010:1889} R.~Picard. \newblock {An Elementary Hilbert Space Approach to Evolutionary Partial   Differential Equations}. \newblock {\em Rend. Istit. Mat. Univ. Trieste}, 42 suppl.:185--204, 2010.
\bibitem{PDE_DeGruyter} R.~Picard and D.~F. McGhee. \newblock {\em Partial Differential Equations: A unified Hilbert Space   Approach}, volume~55 of {\em {De Gruyter Expositions in Mathematics}}. \newblock {De Gruyter. Berlin, New York. 518 p.}, 2011.
\bibitem{Picard2001} Rainer Picard, Norbert Weck, and Karl-Josef Witsch. \newblock {Time-harmonic Maxwell equations in the exterior of perfectly   conducting, irregular obstacles.} \newblock {\em Analysis}, 21:231--263, 2001.
\bibitem{0022-3727-41-15-155412} Daniel Sj{\"o}berg. \newblock {A modified Drude-Born-Fedorov model for isotropic chiral media,   obtained by finite scale homogenization}. \newblock {\em Journal of Physics D: Applied Physics}, 41(15):155412, 2008.
\bibitem{1062.78004} Ioannis~G. Stratis and Athanasios~N. Yannacopoulos. \newblock {Electromagnetic fields in linear and nonlinear chiral media: a   time-domain analysis.} \newblock {\em Abstr. Appl. Anal.}, 2004(6):471--486, 2004.
\bibitem{Yosida_6th} K.~Yosida. \newblock {\em Functional Analysis}. \newblock Springer-Verlag; Berlin et al., 6th edition, 1980.
\end{thebibliography}
\end{document}